\documentclass[12pt]{amsart}
\usepackage{a4ams}
\usepackage{amsfonts}
\usepackage{amssymb}
\usepackage{amsmath}
\usepackage[arrow,matrix,curve]{xy}

\usepackage[active]{srcltx}

\usepackage[normalem]{ulem}
\usepackage{color}

\newcommand{\C}{{\mathbb C}}

\newcommand{\Z}{{\mathbb Z}}

\newcommand{\D}{{\mathbb D}}

\newtheorem{theorem}{Theorem}[section]
\newtheorem{lemma}[theorem]{Lemma}

\newtheorem{proposition}[theorem]{Proposition}
\newtheorem{definition}[theorem]{Definition}

\def\cal{\mathcal}

\newcommand{\calf}[0]{{\cal F}}
\newcommand{\call}[0]{{\cal L}}
\newcommand{\calg}[0]{{\cal G}}

\newcommand{\cale}[0]{{\cal E}}
\newcommand{\calk}[0]{{\cal K}}

\newcommand{\calm}[0]{{\cal M}}

\newcommand{\id}[0]{ {\rm id} }

\begin{document}               

\title{A Green--Julg isomorphism for inverse semigroups}
\author[B. Burgstaller]{Bernhard Burgstaller}
\address{Doppler Institute for mathematical physics,
Trojanova 13, 12000 Praha, Czech Republic}
\email{bernhardburgstaller@yahoo.de}
\subjclass{19K35, 20M18, 46L55, 46L80}
\keywords{Green--Julg isomorphism, inverse semigroup, crossed product, $K$-theory}

\begin{abstract}
For every finite unital inverse semigroup $S$ and 
$S$-$C^*$-algebra $A$
we establish an
isomorphism between
$KK^S(\C,A)$
and
$K(A \rtimes S)$.
This extends the classical Green--Julg isomorphism from finite groups to finite inverse semigroups.
\end{abstract}

\maketitle

%

\section{Introduction}

Let $G$ be a compact group and $A$ a $G$-$C^*$-algebra.
The Green--Julg isomorphism by Green and Julg \cite{0461.46044} states that there is an isomorphism between $G$-equivariant $K$-theory
$KK^G(\C,A)$ of $A$
and the $K$-theory of the crossed product $A \rtimes G$, that is, one has $KK^G(\C,A) \cong K(A \rtimes G)$.
This isomorphism plays a fundamental role in operator $K$-theory and
has been successfully extended to other categories than compact groups $G$ as well,
for example compact groupoids 
by J. L. Tu \cite{0939.19001,0932.19005} and compact quantum groups 
by R. Vergnioux \cite{ThesisRoland,1064.46064}.
%
%

In this note we extend the Green--Julg isomorphism to the class of unital finite inverse semigroups $S$
and the universal crossed product by Khoshkam and Skandalis \cite{1061.46047}.
Formally, it look like the classical isomorphism, that is, we have $KK^S(\C,A) \cong K(A \rtimes S)$, see Theorem \ref{corollaryGreen}.
The proof is done as follows.
We have proven in \cite{burgiKKrdiscrete} that there exists a Baum--Connes map
for a certain class of inverse semigroups, including finite inverse semigroups,
by translating inverse semigroup equivariant $KK$-theory to groupoid equivariant $KK$-theory
and then applying
the Baum--Connes map for groupoids
by Tu \cite{0932.19005}.
Since $S$ is finite, this Baum--Connes map is an isomorphism
$\widehat{KK^S}(\C \rtimes E, A \rtimes E) \cong K(A \rtimes S)$, where $E$ denotes the set of idempotent elements of $S$,
and $\widehat{KK^S}$ compatible $S$-equivariant $KK$-theory \cite{burgiKKrdiscrete}.
Our main work in this note is to establish an isomorphism $\delta^S: KK^S(A,B) \rightarrow \widehat{KK^S}(A \rtimes E,B \rtimes E)$
between 
$S$-equivariant $KK$-theory \cite{burgiSemimultiKK} and compatible $S$-equivariant $KK$-theory \cite{burgiKKrdiscrete}
in Theorem \ref{theoremdelta},
from which the announced Green--Julg isomorphism follows in Theorem \ref{corollaryGreen}. 

This note is organized as follows. In Section \ref{section2} we recall some basic definitions of $S$-equivariant $KK$-theory.
In Section \ref{section3} we define the compatible internal tensor product of $S$-equivariant Hilbert bimodules.
In Section \ref{section4}, Proposition \ref{propositionequivalenceCD}, we show
a certain equivalence of categories between compatible and incompatible $S$-equivariant Hilbert bimodules.
In Section \ref{section5}, we use this to prove that $\delta^S$ is an isomorphism, see Theorem \ref{theoremdelta},
and deduce the Green--Julg isomorphism for inverse semigroups in Theorem \ref{corollaryGreen}.

%
%
%

\section{$S$-equivariant $KK$-theory}
\label{section2}

Let $S$ be a unital inverse semigroup and $E$ its subset of idempotent elements.
We recall here some basic definitions about $C^*$-algebras and $KK$-theory in the $S$-equivariant case,
see \cite{burgiSemimultiKK} or \cite{burgiDescent}. 
Because $S$ is an inverse semigroup, these definitions become slightly compacter than in \cite{burgiDescent}.
We shall always assume that the unit $1 \in S$ acts as the identity on the respective category.

\begin{definition}   \label{defCstar}
{\rm
An {\em $S$-Hilbert $C^*$-algebra} is a $\Z/2$-graded $C^*$-algebra $A$ with a
unital semigroup homomorphism
$\alpha: S \rightarrow \mbox{End}(A)$ such that $\alpha_s$ respects the grading
and $\alpha_{s s^{*}}(x) y = x \alpha_{s s^{*}}(y)$
for all $x,y \in A$ and $s \in S$.
}
\end{definition}


We usually write $s(a):=\alpha_s(a)$ for the action of $S$ on $A$.
%
A $*$-homomorphism $f:A \rightarrow B$ between $S$-Hilbert $C^*$-algebras $A$ and $B$ is called {\em $S$-equivariant}
if $f(s(a))= s(f(a))$ for all $a \in A$ and $s \in S$.
We regard the class of $S$-Hilbert $C^*$-algebras as a category where the morphisms are the $S$-equivariant $*$-homomorphisms.

\begin{definition}   \label{defSHilbertModule}
{\rm
Let $B$ be an $S$-Hilbert $C^*$-algebra. An {\em $S$-Hilbert $B$-module $\cale$} is a $\Z/2$-graded Hilbert $B$-module
which is equipped with a unital semigroup homomorphism $U: S \rightarrow {\rm Lin}(\cale)$ (linear maps on $\cale$) such that
$U_s$ respects the grading, $U_{s s^*}$ is a self-adjoint projection in $\call(\cale)$, and the identities
$\langle U_s(\xi),\eta \rangle = s(\langle \xi, U_{s^*} (\eta) \rangle )$ and
$U_s(\xi b) = U_s(\xi) s(b)$ hold for all $s \in S$ and $\xi,\eta \in \cale$.
}
\end{definition}


\begin{definition}   \label{SequivariantRep}
{\rm
Let $A$ and $B$ be $S$-Hilbert $C^*$-algebras.
An {\em $S$-Hilbert $A,B$-bimodule} $\cale$ is an $S$-Hilbert $B$-module $\cale$ with a $*$-homomorphism
$\pi:A \rightarrow \call(\cale)$, which is an {\em $S$-equivariant representation} of $A$ on $\cale$
in the sense that
$U_s \pi(a) U_{s^*} = \pi(s(a)) U_{s} U_{s^*}$ and $[U_s U_{s^*},\pi(a)]=0$
(commutator) for all $a \in A$ and $s \in S$.
}
\end{definition}

We often write $a \xi$ rather than $\pi(a) \xi$.
We remark that $A$ and $B$ act usually {\em incompatibly} on $\cale$ in the sense that $U_e(\xi) b \neq \xi U_e(b)$ for $e \in E$ and $b \in B$,
and similarly so on the $A$-side.
%
We shall consider compatible versions of Hilbert bimodules in the following sense.

\begin{definition}
{\rm
We call an $S$-Hilbert $A,B$-bimodule $\cale$ {\em compatible} if $e(a) \xi = a U_e(\xi)$ and $U_e(\xi) b = \xi e(b)$ for all $e \in E, \xi \in \cale, a \in A$ and $b \in B$.
}
\end{definition}

A {\em morphism $\mu:\cale \rightarrow \calf$} between $S$-Hilbert $A,B$-bimodules $\cale$ and $\calf$ is
understood to strictly respect all involved structures
on both sides
(including the $B$-valued inner product).
%
We view the class of $S$-Hilbert $A,B$-bimodules together with these morphisms as a category. 
It forms a subclass of the class of all (sometimes called {\em incompatible})
$S$-Hilbert $A,B$-bimodules.




\begin{definition}   \label{defKK}
{\rm
An {\em $S$-equivariant $A,B$-cycle} $(\cale,T)$ consists of an $S$-Hilbert $A,B$-bimodule $\cale$ and an operator $T \in \call(\cale)$ such that
$(\cale,T)$ is a non-equivariant cycle in the sense of Kasparov (\cite{kasparov1981,kasparov1988})
and both $[U_{s s^*},T]$ and $U_s T U_{s^*} - U_{s s^*} T$
are in $\{ S \in \call(\cale)|\, a S , S a \in \calk(\cale)\}$
for all $s \in S$.
$KK^S(A,B)$ is defined to be the class of $S$-equivariant $A,B$-cycles divided by homotopy induced by $S$-equivariant $A,B[0,1]$-cycles.
}
\end{definition}

%
A cycle is called {\em compatible} if the underlying $S$-Hilbert $A,B$-bimodule is compatible.
$\widehat{KK^S}(A,B)$ is defined to be the set of compatible cycles divided by homotopy (induced by compatible cycles),
see also \cite[Section 3]{burgiDescent}.
%
%
The {\em full crossed product} $A \rtimes S$ of an $S$-Hilbert $C^*$-algebra $A$ (see \cite{1061.46047})
is the enveloping $C^*$-algebra of
the involutive Banach algebra
$\ell^1(S,A):= \{ a: S \rightarrow A| \, a_s \in A_{s s^*} := s s^*(A), \, \sum_{s \in S} \|a_s\| < \infty \}$ under convolution
$(\sum_{s \in S} a_s \rtimes s)(\sum_{t \in S} b_t \rtimes t) := \sum_{s,t \in S} a_s s(b_t) \rtimes st$ and
involution $(\sum_{s \in S} a_s \rtimes s)^* := \sum_{s \in S} s^*(a_s^*) \rtimes s^*$
(standard elements of $A \rtimes S$ are denoted by $a \rtimes s$).
%
%
Sieben's crossed product \cite{sieben1997} is denoted by $A \widehat \rtimes S$.

\section{The compatible internal tensor product}
\label{section3}

For the rest of the paper we assume that $E$ is a finite set.
The commutative $C^*$-algebra $C^*(E)$ freely generated by the set $E$ of commuting projections (see \cite{1061.46047})
may be identified with the full crossed product $\C \rtimes E$,
and with $C_0(X)$, where $X$ denotes the (finite) spectrum of $C^*(E)$. The canoncial generators of $C^*(E)$
are denoted by $u_e \in C^*(E)$ ($e \in E$).
By universality, $C^*(E)$ induces a $*$-homomorphism $C_0(X) \rightarrow Z(\calm(A)): u_e \mapsto \alpha_e$ (center of the multiplier algebra)
for every $S$-Hilbert $C^*$-algebra $(A,\alpha)$.
Similarly we have a $*$-homomorphism $C_0(X) \rightarrow \call(\cale): u_e \mapsto U_e$ for every $S$-Hilbert module $\cale$.

%

Every minimal projection $P$ in $C^*(E)$ corresponds to an element $x \in X$ such that $P=1_{\{x\}}$ in $C_0(X)$,
and thus we loosely write $P \in X$ for a minimal projection $P$.
%
It may be written as
\begin{equation} \label{expressionP}
P = u_{e_1} \ldots u_{e_n} (1- u_{f_1}) \ldots (1-u_{f_m}),
\end{equation}
where $E=\{e_1,\ldots,e_n\} \sqcup \{f_1, \ldots,f_m\}$ is some partition of $E$ into two parts ($n \ge 1, m \ge 0$).
We also have $P=u_e (1- u_{e f_1}) \ldots (1-u_{e f_m})$, where $e=e_1 \ldots e_n$. In this way
we see that every $P$ can be written in so-called {\em standard form}
$P= u_e \prod_{f \in E, f < e} (1-u_f)$ with $e \in E$, and
vice versa, every expression in standard form is an element of $X$.
Define $E_e:= e E$ for $e \in E$.
For all $s \in S$ there is an order preserving isomorphism $\gamma_s: E_{s^* s} \rightarrow E_{s s^*}$ (with inverse $\gamma_{s^*}$)
defined by $\gamma_s(e) = s e s^*$.

For general considerations, we enlarge the set of letters $u_e$ ($e \in E$) by considering also formal letters
$u_s$ for every $s \in S$. 
In practical terms we mean by $u$
a formal $S$-action which is not specified, and which has to be
replaced by the concrete $S$-action when applied to concrete Hilbert $C^*$-algebras and modules.
Note that we have $u_s P = P_s u_s$
for $s \in S$, where
\begin{equation} \label{Ps}
P_s :=  u_s P u_{s^*} = u_{s e_1 s^*} \ldots u_{s e_n s^*} (1 -
u_{s f_1 s^*}) \ldots (1- u_{s f_m s^*}).
\end{equation}

Let $\cale$ and $\calf$ be incompatible $S$-Hilbert bimodules with $S$-actions $U$ and $V$, respectively.
Define the self-adjoint diagonal projection
$\D := \sum_{P \in X} P \otimes P$
on the internal tensor product $\cale \otimes_B \calf$.
(Recall from \cite{burgiSemimultiKK} that $U_e \otimes 1$ and $1 \otimes V_e$ are well-defined self-adjoint projections on $\cale \otimes_B \calf$.)


\begin{lemma}  \label{lemmaDcommutes}
$\D$ commutes with $u_s \otimes u_s$.
\end{lemma}

\begin{proof}
Write $P= u_e \prod_{f < e}(1-u_f)$ in standard form. Let $s \in S$. Since $P$ is a minimal projection,
either $u_s P = u_s u_{s^* s} P=0$ or $P \le u_{s^* s}$.
In the latter case we have $e \in E_{s^* s}$ and
$$u_s P = P_s u_s = u_{s e s^*} \prod_{f < e} (1- u_{s f s^*}) u_s
= u_{\gamma_s(e)} \prod_{f < e}(1-u_{\gamma_s(f)}) u_s.$$
We see here that $P_s$ is again in standard form as $\gamma_s$ is an order isomorphism.
Setting $Y_g = \{P \in X| \, P \le u_g\}$,
the map $P \mapsto P_s$ defines a bijection $Y_{s^* s } \rightarrow Y_{s s^*}$ (with inverse map $P \mapsto P_{s^*}$).
Consequently,
\begin{eqnarray*}
(u_s \otimes u_s) \D &=& \sum_{P \in Y_{s^* s}} u_s P \otimes u_s P = \sum_{P \in Y_{s^* s}} P_s u_s \otimes P_s u_s
= \D (u_s \otimes u_s).
\end{eqnarray*}
\end{proof}

From the arguments in the last proof we also get the following corollary.

\begin{lemma} \label{lemmaPs}
(i) For every $P \in X$ and $s \in S$ one has
$P \le u_{s^* s}$ iff $P_s \neq 0$ iff $P_s \in X$. 

(ii) If $P$ is in standard form in (\ref{expressionP}) then $P_s$ is also in
standard form in (\ref{Ps}).

(iii) Denoting $Y_e = \{P \in X| \, P \le u_e\}$ for $e \in E$,
the map $P \mapsto P_s$ defines a bijection $Y_{s^* s } \rightarrow Y_{s s^*}$.

%
\end{lemma}

\begin{definition}   \label{definitionCompatibleTensor}
{\rm
Let $\cale$ an $S$-Hilbert $A,B$-bimodule and $\calf$ an $S$-Hilbert $B,C$-bimodule.
The {\em compatible} internal tensor product $\cale \otimes_B^{X} \calf$ is the sub-$S$-Hilbert
$A,C$-bimodule $\D(\cale \otimes_B \calf)$ of $\cale \otimes_B \calf$.
}
\end{definition}

By Lemma \ref{lemmaDcommutes}, $\cale \otimes_B^{X} \calf$ is invariant under the $S$-action.
The tensor product $\cale \otimes_B^{X} \calf$ is now compatible, that is,
$\D(u_e(\xi) \otimes \eta) =\D(\xi \otimes u_e(\eta))$ 
for every $e \in E$.
%
If the module multiplication between $\cale$ and $B$ is
compatible then it is also compatible with respect to $P$, that is, $P(\xi) b = \xi P(b) = P(\xi b) = P(\xi)
P(b)$ (by induction with expression (\ref{expressionP})).


\section{A categorial equivalence}
\label{section4}


%
Given an $S$-Hilbert $C^*$-algebra $A$, we regard the crossed product $A \rtimes E$ as a $S$-Hilbert $C^*$-algebra
under the $S$-action $\beta_s(a \rtimes e) = s(a) \rtimes s e s^*$, see \cite{1061.46047}.
We let $A$ act on $A \rtimes E$ by multiplication,
which 
is an $S$-equivariant representation $A \rightarrow \call(A \rtimes S)$ in the sense of
Definition \ref{SequivariantRep}.
%
%
Throughout
let us now fix two $S$-Hilbert $C^*$-algebras $A$ and $B$. The category of (incompatible) $S$-Hilbert $A,B$-bimodules
is denoted by $\mathsf{C}$,
and the category of compatible $S$-Hilbert $A \rtimes E,B \rtimes E$-bimodules by $\mathsf{D}$.


\begin{lemma}  \label{lemmafunctorC}
We have a functor $\mathsf{F}:\mathsf{C} \rightarrow \mathsf{D}$ given by $\mathsf{F}(\cale) = \cale \otimes_B^{X} (B \rtimes E)$
for objects $\cale$ in $\mathsf{C}$, and $\mathsf{F}(\mu) = (\mu \otimes 1) \D
:= (\mu \otimes \mbox{id}_{B \rtimes E})|_{\mathsf{F}(\cale)}$
for morphisms $\mu$ in $\mathsf{C}$.
\end{lemma}

\begin{proof}
We turn $\cale \otimes_B (B \rtimes E)$ into a left $A \rtimes E$-module via the
compatible $S$-equivariant representation
$$\overline \pi: A \rtimes E \longrightarrow \call(\cale \otimes_B (B \rtimes E)): \overline \pi(a \rtimes e) = (\pi(a) \otimes 1) \overline U_e,$$
where
$\overline U_s = U_s \otimes \beta_s$ denotes the diagonal action, see also \cite[Lemma 6.5]{burgiKKrdiscrete}.
%
Clearly, $\D$ commutes with $\pi(a) \otimes 1$ and $\overline U_e$ and so with $\overline \pi(a \rtimes e)$.
Hence $\D(\cale \otimes_B (B \rtimes E))$ is an $S$-Hilbert $A \rtimes E,B\rtimes E$-bimodule.
It remains to check that it is compatible on the $B \rtimes E$-side.
Let $x$ and $y$ in $B \rtimes E$ and $e \in E$.
Then 
we have
\begin{eqnarray*}
\D(\xi \otimes x) u_e(y) &=& \D(\xi \otimes x u_e (y)) = \D(\xi \otimes u_e(xy))
= \D(u_e(\xi) \otimes u_e(x) y)\\
&=& ((u_e \otimes u_e)\D(\xi \otimes x)) y
\end{eqnarray*}
by Lemma \ref{lemmaDcommutes}.
%
For a morphism $\mu:\cale \rightarrow \cale'$ between $S$-Hilbert bimodules, $\mu \otimes 1$ commutes with $\D$ on $\cale \otimes_B \calf$,
and so $(\mu \otimes 1)\D$ is a map on the compatible tensor product.
\end{proof}

For simplicity, we will from now on assume that $B$ has a unit $1_B$. If $B$ has not a unit,
one replaces it by an approximate unit and takes the limit along the approximate unit in all expressions there where $1_B$
appears in the text.
Given $P$ as in (\ref{expressionP}), denote by $\rho_P: B \rightarrow B$ the projection $u_{e_1} \ldots u_{e_n}$ acting on $B$.

\begin{lemma}  \label{lemmauniquexp}
For every $x \in B \rtimes E$ and $P \in X$ there exists a unique $x_P \in \rho_P(B)$ such that $x P(1_B \rtimes 1) = x_P P(1_B \rtimes 1)$.
The map $\sigma_P : B \rtimes E \rightarrow B$, $\sigma_P(x) = x_P$, is a $*$-homomorphism.
\end{lemma}

\begin{proof}
In every $S$-Hilbert $C^*$-algebra the multiplication is compatible and so in particular one has $P(a) b = a P(b)$.
We choose for $P$ a representation as in (\ref{expressionP}). Let $x= b \rtimes g$ be an elementary element in $B \rtimes E$.
Then $(b \rtimes g) P(1_B \rtimes 1) = P(b \rtimes g) (1_B \rtimes 1)= P(b \rtimes g)$. Either $g = f_i$ for some $i$, in which case
$$P(b \rtimes g)= P\cdot (1-u_{f_i}) (b \rtimes f_i) = 0 = 0 P(1_B \rtimes 1),$$
or
$g = e_i$ for some $i$, in which case we choose the standard form $P=u_e \prod_{f < e}(1-u_f)$ for $P$ and get
\begin{eqnarray*}
&& P(b \rtimes g) = P(b \rtimes 1) = (b \rtimes 1) P(1_B \rtimes 1) = b (u_e(1_B) \rtimes 1 + \ldots)\\
&=&  u_e(b) u_e(1_B) \rtimes 1 + \ldots
= u_e(b) (u_e(1_B) \rtimes 1 + \ldots) = \rho_P(b) P(1_B \rtimes 1)
\end{eqnarray*}
by expansion of $P(1_B \rtimes 1)$.
So in either case we have found some $x_P \in \rho_P(B)$ satisfying the claimed identity.
On the other hand, any such $x_P$ is unique because in the expansion $x_P P(1_P \rtimes 1) = x_P \rtimes e + \ldots$
the factor $x_p \rtimes e$ is linearly independent from the other factors in the expansion, and so is unique.
We have checked that $\sigma_P(b \rtimes g) = 0$ if $g \le f_i$, and $\sigma_P(b \rtimes g)= \rho_P(b)$ if $g=e_i$, and with that one easily
verifies 
that $\sigma_P$ is a $*$-homomorphism. 
%
\end{proof}

We are going to describe how we may associate incompatible Hilbert bimodules to compatible Hilbert bimodules.




\begin{lemma}  \label{lemmafunctorG}
There is a functor $\mathsf{G}:\mathsf{D} \rightarrow \mathsf{C}$ defined by $\mathsf{G}(\calf): = \cale$
for objects $\calf$ in $\mathsf{D}$ and $\mathsf{G}(\mu) := \mu$ for morphisms $\mu$ in $\mathsf{D}$, where
$\cale$ is defined to be identical to $\calf$ as a graded vector space with the same $S$-action as $\calf$,
and the Hilbert $A,B$-bimodule structure on $\cale$ is
defined by $\xi \cdot b := \xi (b \rtimes 1)$, $a \cdot \xi   := (a \rtimes 1) \xi$, and
\begin{eqnarray}
\langle P \xi , P \eta
\rangle_\cale P(1_B \rtimes 1) &:=& \langle P \xi ,  P \eta \rangle_\calf   \label{innerprodEqu} 
\end{eqnarray}
for all $a \in A, b \in B, \xi,\eta \in \calf$ and $P \in X$, where
$\langle P \xi , P \eta  \rangle_\cale$ in (\ref{innerprodEqu}) has to be chosen to be the unique $x_P \in B$ of Lemma \ref{lemmauniquexp}
for $x =\langle P \xi ,  P \eta \rangle_\calf \in B \rtimes E$.
Moreover, we have $\call(\calf) \subseteq \call(\cale)$ canonically.
%

\end{lemma}

\begin{proof}
Notice that $\langle P \xi ,  P \eta  \rangle_\calf = \langle P \xi,  P \eta \rangle_\calf P(1_B \rtimes 1)$
in (\ref{innerprodEqu})
by compatibility of $\calf$.
The inner product on
$\cale$ is determined by (\ref{innerprodEqu}) and
$\langle \xi, \eta \rangle_\cale =
\Big \langle \sum_{P \in X} P \xi, \sum_{Q \in X} Q \eta \Big \rangle_\cale
:= \sum_{P \in X} \langle P \xi,  P \eta \rangle_\cale.$
(The last identity is necessary since $u_e$, and consequently $P$, need to act as self-adjoint projections on $\cale$.)
We are going to check that the inner product on $\cale$ respects the $B$-module multiplication.
We have
\begin{eqnarray*}
\langle P \xi , P(\eta \cdot b)  \rangle_\cale P(1_B \rtimes 1) &=& \langle P \xi ,  P(\eta(b \rtimes 1)) \rangle_\calf
= \langle P \xi ,  P \eta  \rangle_\calf P(b \rtimes 1)\\
&=& (\langle P \xi , P \eta  \rangle_\cale b) P(1_B \rtimes 1).
\end{eqnarray*}
By the uniqueness of $x_P$ in Lemma \ref{lemmauniquexp} we get $\langle P \xi, P(\eta \cdot b)  \rangle_\cale
= \langle P \xi, P \eta \rangle_\cale b$. (Note that $\rho_P(a) b = \rho_P(ab)$.)
Writing
\begin{equation} \label{innerproductEsigmaP}
\langle \xi, \eta \rangle_\cale = \sum_{P \in X} \langle P \xi, P \eta \rangle_\cale = \sum_{P \in X} \sigma_P (\langle P \xi, P \eta \rangle_\calf),
\end{equation}
we easily see that we have here indeed a (positive definite) $B$-valued inner product on $\cale$ because $\sigma_P$ is a $*$-homomorphism by Lemma \ref{lemmauniquexp}.

We aim to show that $\call(\calf) \subseteq \call(\cale)$, which proves the last claim of the lemma. Let $T \in
\call(\calf)$ and $T^*$ its adjoint in $\call(\calf)$.
Since the module
multiplication $\calf \times B \rightarrow \calf$ is
compatible, $ T P= PT$ for all $P \in X$.
Hence, by (\ref{innerproductEsigmaP}) we get $\langle T \xi,\eta \rangle_\cale = \langle \xi, T^* \eta \rangle_\cale$,
proving the claim.
This also
shows that $A$ acts via adjoint-able operators on $\cale$, and we easily see that $A$ acts through a $*$-homomorphism on $\call(\cale)$.
We now focus on the $S$-action. We have, for $\xi \in \cale$,
\begin{eqnarray*}
U_s(\xi \cdot b) &=& U_s(\xi) s(b \rtimes 1) = U_s(\xi) (s(b) \rtimes s s^*) = U_s(\xi) s s^* (s(b) \rtimes 1)\\
&=&  U_s(\xi) (s(b) \rtimes 1)
= U_s(\xi) \cdot s(b),
\end{eqnarray*}
proving one identity of Definition \ref{defSHilbertModule}.
The operator
$U_s U_{s^*}$ is self-adjoint on $\calf$ and so also on $\cale$.
It is easy to check that $A$ acts as an $S$-equivariant representation on $\cale$ (Definition \ref{SequivariantRep}).

We are going to show that $\langle U_s \xi, \eta \rangle_\cale = s \langle \xi, U_{s^*} \eta \rangle_\cale$
(Definition \ref{defSHilbertModule}).
%
%
Assume that $P\in X$, $P_{s^*} \neq 0$, $b \in B$ and $b=\rho_{P_{s^*}}(b)$. 
By Lemma \ref{lemmaPs}, $P \le s s^*$ and so $e \le s s^*$ since $e=e_1 \ldots e_n$ in identity (\ref{expressionP})
and $e_i = ss^*$ for some $i$.
Then,
writing $P=u_e\prod_{f < e}(1-u_f)$ in standard form and expanding it, we get
\begin{eqnarray} \label{specialsb}
&& s(bP_{s^*}(1_B \rtimes 1)) = e s(b) \rtimes s s^* e s s^* + \ldots  
= s(b) P(1_B \rtimes 1),\\
&& \rho_{P}(s(b))= u_{e}(s(b)) = u_{s s^* e s}(b) = s (\rho_{P_{s^*}}(b)) = s(b)  \label{specialsb2}
\end{eqnarray}
by Lemma \ref{lemmaPs}.
Note that $P u_s = u_s P_{s^*}$.
Hence,
\begin{eqnarray*}
&& \langle P U_s \xi, P \eta\rangle_\cale P(1_B \rtimes 1) = \langle
P U_s  \xi ,P \eta \rangle_\calf = s \langle P_{s^*} \xi , P_{s^*}
U_{s^*} \eta \rangle_\calf \\
&=& s \big (\langle P_{s^*} \xi, P_{s^*} U_{s^*} \eta \rangle_\cale P_{s^*}(1_B
\rtimes 1) \big )  = s \big (\langle P_{s^*} \xi, P_{s^*} U_{s^*} \eta
\rangle_\cale \big ) P (1_B \rtimes 1),
\end{eqnarray*}
where the last identity is by (\ref{specialsb}) for $b:= \langle P_{s^*} \xi, P_{s^*} U_{s^*} \eta \rangle_\cale \in \rho_{P_{s^*}}(B)$.
By (\ref{specialsb2}), the last computation and the uniqueness of $x_P$,
we get
$\langle P U_s \xi, P \eta\rangle_\cale = s \langle
P_{s^*}\xi, P_{s^*} U_{s^*} \eta \rangle_\cale$.
If $P_{s^*} = 0$ then the last identity is trivially also true.
Hence it follows
$$\langle U_s \xi, \eta \rangle_\cale = \sum_{P \in X} \langle P U_s  \xi, P \eta \rangle_\cale
= \sum_{P \in X} s \langle P_{s^* }\xi, P_{s^*} U_{s^*} \eta
\rangle_\cale = s \langle \xi, U_{s^*} \eta \rangle_\cale,$$
where the last identity is by Lemma \ref{lemmaPs}.(iii). Hence, $U$ is evidently an $S$-action on $\cale$.
One easily checks that $\mathsf{G}(\mu)$ is a morphism if $\mu$ is a morphism (note that we have
$\call(\calf) \subseteq \call(\cale)$).
%
\end{proof}

\begin{proposition}   \label{propositionequivalenceCD}
The functors $\mathsf{F}$ and $\mathsf{G}$ define a categorial equivalence between $\mathsf{C}$ and
$\mathsf{D}$.
%
\end{proposition}

\begin{proof}
We need to show that the functor $\mathsf{G} \mathsf{F}$ is naturally equivalent to $\id_{\mathsf{C}}$ by a natural
isomorphism $\kappa$, and the functor $\mathsf{F} \mathsf{G}$ to $\id_{\mathsf{D}}$ by a natural
isomorphism $\lambda$.
Let $\calf$ be an object in $\mathsf{D}$.
Let us denote the
map $\calf \rightarrow \mathsf{G}(\calf)$
by $W$ for greater clarity (for any $\calf$), although it is regarded as an identical map on sets.
We define $\kappa_\cale: \cale \rightarrow \mathsf{G} \mathsf{F}(\cale)$ and $\lambda_\calf:\calf \rightarrow \mathsf{F} \mathsf{G}(\calf)$
by $\kappa_\cale = W \circ f$ and $\lambda_\calf = f \circ W$, where $f$ denotes the linear isomorphism
\begin{equation}  \label{mapf}
f:\cale \longrightarrow \cale \otimes_B^{X} (B \rtimes E):f(\xi) = \D(\xi \otimes (1_B \rtimes 1)).
\end{equation}
It is indeed surjective, as
$$\D(\xi \otimes (b \rtimes e)) = \D(\xi b \otimes u_e(1_B \rtimes 1)) = \D(u_e(\xi b) \otimes (1_B \rtimes 1))$$
for any given $\xi \in \cale, e\in E$ and $b \in B_e$.
For $s \in S$ we have
\begin{eqnarray*}
f(U_s \xi ) &=& \D \big (U_s \xi  \otimes (1_B \rtimes 1) \big) = \D \big (U_{s s^* }U_s(\xi)s(1_B) \otimes (1_B \rtimes 1) \big )\\
& =& \D \big (U_s \xi \otimes s s^* ( s(1_B) (1_B \rtimes 1)) \big )  =
\D \big (U_s \xi \otimes s(1_B \rtimes 1) \big ) 
= u_s f(\xi) 
\end{eqnarray*}
by Lemma \ref{lemmaDcommutes}.
Hence $f$, and so $\kappa_\cale$ and $\lambda_\calf$ respect the $S$-action.
We have $f(a \xi) = (a\rtimes 1)f(a)$, and thus $\kappa_\cale(a \xi)= W(a \rtimes 1 f(a))= a \cdot W(f(a))= a \cdot \kappa_\cale(a)$,
and
$$\lambda_\calf((a\rtimes e)\eta) = \lambda_\calf(u_e(a \rtimes 1)\eta)= \lambda_\calf((a\rtimes 1)u_e(\eta))= (a \rtimes 1)u_e(\lambda_\calf(\eta))
= (a \rtimes e)\lambda_\calf(\eta).$$
Hence $\kappa_\cale$ and $\lambda_\calf$ respect the left module structure. They automatically also respect the right module structure as we are even going to show
that they are unitary operators.
We have, for $\xi, \eta \in \calf$, and by omitting notating the map $W$,
\begin{eqnarray*}   
&&\langle P \lambda_\calf \xi , P \lambda_\calf \eta \rangle =
\langle P f  \xi, P f \eta \rangle
= \langle P \xi \otimes P(1_B \rtimes 1),P \eta \otimes P(1_B \rtimes 1)\rangle \\
&=& \langle P \xi , P \eta \rangle 
P(1_B \rtimes 1) = \langle P \xi, P \eta \rangle_\calf,
\end{eqnarray*}
and so $\lambda_\calf$ is evidently a unitary operator.
Similarly we have, for $\xi,\eta \in \cale$,
\begin{eqnarray*}
&&\langle P \kappa_\cale \xi , P \kappa_\cale \eta \rangle P(1_B \rtimes 1)
= \langle P f \xi , P f \eta \rangle
= \langle P \xi , P \eta \rangle_\cale P(1_B \rtimes 1),
\end{eqnarray*}
which also shows that $\kappa_\cale$ is a unitary operator by the uniqueness of the coefficient $x_P$ (Lemma \ref{lemmauniquexp}).
For morphisms $\mu$ we have $\mathsf{G} \mathsf{F}(\mu) = (\mu \otimes 1)\D$ and $\mathsf{F} \mathsf{G}(\mu) = (\mu \otimes 1)\D$.
So we get $\mathsf{G} \mathsf{F}(\mu) \circ \kappa_\cale = \kappa_{\cale'} \circ \mu$ and $\mathsf{F} \mathsf{G}(\mu) \circ \lambda_\calf = \lambda_{\calf'} \circ \mu$,
which completes the proof.
\end{proof}


\section{The Green--Julg isomorphism}
\label{section5}

\begin{lemma}  \label{lemmadelta}
There is a homomorphism
$\delta^S: KK^S(A,B) \longrightarrow \widehat{KK^S}(A \rtimes E, B \rtimes E)$
defined by
$\delta^S(\cale,T) = \big (\cale \otimes_B^X (B\rtimes E), \D (T \otimes 1) \D \big )$	 
on cycles.
\end{lemma}

\begin{proof}
There is a homomorphism $\epsilon:KK^S(A,B) \longrightarrow KK^S(A \rtimes E,B \rtimes E)$ defined
by $\epsilon^S(\pi, \cale,T) = (\overline \pi, \cale \otimes_B (B\rtimes E),T\otimes 1)$ on cycles
by \cite[Theorem 6.7.(a)]{burgiKKrdiscrete} and the remark after that theorem.
The action $\overline \pi$ of $A \rtimes E$ on $\cale \otimes_B (B\rtimes E)$ commutes with $\D$
(see the proof of Lemma \ref{lemmafunctorC}).
The cycle $\delta^S(\cale,T)$ is then just the cycle $\epsilon^S(\pi,\cale,T)$ cut down by the projection $\D$,
so is a cycle again. More precisely, for example, to check Definition \ref{defKK}, one has (with Lemma \ref{lemmaDcommutes}
and Definition \ref{SequivariantRep})
\begin{eqnarray*}
&& (a \rtimes e) \big ((u_s \otimes u_s) \D (T \otimes 1) \D (u_{s^*} \otimes u_{s^*}) - (u_{s s^*} \otimes u_{s s^*}) \D (T \otimes 1) \D \big )\\
&=& \D (u_e \otimes u_e)(1 \otimes u_{s s^*}) \big( a (u_s T  u_{s^*} - u_{s s^*} T) \otimes 1 \big ) \D \\
&=& \D (u_e \otimes u_e)(1 \otimes u_{s s^*})( k \otimes 1) \D,
\end{eqnarray*}
and this is a compact operator on $\D (\cale \otimes_B (B \rtimes E)) \D$ because $k := a (u_s T  u_{s^*} - u_{s s^*} T)$
is in $\calk(\cale)$ by Definition \ref{defKK}.
We omit the straightforward proof that $\delta^S$ respects homotopy.
%
%
%
\end{proof}

%

\begin{lemma}  \label{lemmagamma}
There exists a homomorphism
$\gamma^S: \widehat{KK^S}(A \rtimes E, B \rtimes E) \longrightarrow KK^S(A,B)$
defined by
$\gamma^S(\calf,T) = (\mathsf{G}(\calf), T)$	
on cycles.
\end{lemma}

\begin{proof}
By the last assertion of Lemma \ref{lemmafunctorG}, the identical map $\iota: T \mapsto T$ sends $\call(\calf)$ into $\call(\mathsf{G}(\calf))$.
%
%
We aim to show that $\iota$ maps compact operators to compact operators.
Let $\theta_{\xi,\eta} \in \call(\calf)$ denote the elementary
compact operator $\zeta \mapsto \xi \langle \eta, \zeta \rangle_\calf$. Set $\cale = \mathsf{G}(\calf)$. We have
\begin{eqnarray*}
\xi \langle \eta, \nu \rangle_\calf  &=&  \sum_{P \in X} P(\xi) \langle P  \eta, \nu
\rangle_\calf = \sum_{P \in X} P(\xi) \langle P \eta, \nu \rangle_\cale P(1_B \rtimes 1) \\
&=& \sum_{P \in X} P(\xi) (\langle P \eta, \nu \rangle_\cale \rtimes 1)
= \sum_{P \in X} P(\xi) \cdot \langle P \eta, \nu \rangle_\cale,
\end{eqnarray*}
which is a compact operator in $\call(\cale)$, where $\cdot$ denotes the $B$-module multiplication in $\cale$.
It is not difficult to check that $\gamma^S(\calf,T)$ is a cycle.

It remains to check that $\gamma^S$ respects homotopy.
Let
$(\calf,T)$ be a homotopy,
so an $A \rtimes E, (B \rtimes E)[0,1]$-Kasparov cycle,
%
and denote
$(\cale,T) := \gamma^S(\calf,T)$. (Recall that $\cale$ is an identical copy of $\calf$ as a set.)
Write
$\varphi_t: (B \rtimes E)[0,1] \rightarrow B \rtimes E$ and $\psi_t:
B[0,1] \rightarrow B$ for the evaluation maps at time $t$, and denote by
$\calf_t = \calf \otimes_{\varphi_t} (B \rtimes E)$ and $\cale_t =
\cale \otimes_{\psi_t} B$ evaluation of $\calf$ and $\cale$ at time $t$.
We aim to show that $\cale$ is a homotopy connecting $(\overline \cale_0,T \otimes 1)$ with $(\overline \cale_1,T \otimes 1)$,
where we denote $(\overline \cale_t,T \otimes 1) := \gamma^S(\calf_t,T \otimes 1)$.
To this end it is enough to show that $\omega: \cale_t \rightarrow \overline \cale_t$ with
$\omega(\xi \otimes b) = \xi
\otimes (b \rtimes 1)$ is an isomorphism of $S$-Hilbert $A,B$-bimodules, because $(\cale_0,T \otimes 1)$ and $(\cale_1,T \otimes 1)$ are homotopically connected by
$(\cale,T)$.
That $\omega$ respects the $A,B$-bimodule structure is obvious.
Note that $\calf$ is a compatible bimodule and $\varphi$ is a compatible representation. Thus
\begin{equation}  \label{blanacedft}
\xi \otimes b \rtimes e = \xi \otimes u_e(b \rtimes 1) = \xi u_e(1_{(B \rtimes E)[0,1]}) \otimes (b \rtimes 1)
= u_e(\xi) \otimes (b \rtimes 1)
\end{equation}
in $\overline \cale_t$, which shows that $\omega$ is surjective.
Similarly, we see that $\omega$ is $S$-equivariant as
$$\omega(u_s(\xi) \otimes s(b)) = u_{s s^*}u_s(\xi)
\otimes (s(b) \rtimes 1) = u_s(\xi) \otimes s(b \rtimes 1).$$


We have $P(\xi \otimes (b \rtimes 1)) = P(\xi) \otimes P(b \rtimes 1)$ in $\calf_t$.
Hence, the inner product of $\overline \cale_t$ is computed from the one of $\calf_t$
by
\begin{equation}  \label{eft1}
\big \langle P(\xi \otimes (b \rtimes 1)), P(\eta \otimes (c \rtimes
1)) \big \rangle_{\overline \cale_t} P(1_B \rtimes 1)  = P(b \rtimes 1)^* \varphi_t \big (
\langle P \xi, P \eta \rangle_\calf \big )  P(c \rtimes 1).
\end{equation}
By the connection
between the inner products of $\cale$ and $\calf$, we have
$$\psi_t \big ( \langle P \xi, P \eta \rangle_\cale \big ) P(1_B \rtimes 1) =
\varphi_t \big( \langle P \xi, P \eta \rangle_\cale P(1_{B[0,1]} \rtimes 1) \big ) =
\varphi_t \langle P \xi, P \eta\rangle_\calf.$$
Multiplying here from
the left and right with $P(b^* \rtimes 1)$ and $P(c \rtimes 1)$,
respectively, gives
\begin{eqnarray}  \label{eft2}
\big (b^* \psi_t( \langle P \xi, P \eta \rangle_\cale) c \big ) P(1_B \rtimes
1) &=&
P(b^* \rtimes 1) \varphi_t \big (\langle P \xi, P \eta\rangle_\calf \big )
P(c \rtimes 1).
\end{eqnarray}
By a similar computation as in (\ref{blanacedft}) we get $u_e \otimes u_e = u_e \otimes 1$ on $\cale_t$, and so also
$P \otimes P = P \otimes 1$ on $\cale_t$.
Consequently, the left hand side of (\ref{eft2}) equals $\langle P(\xi \otimes b), P(\eta \otimes c) \rangle_{\cale_t} P(1_B \rtimes 1)$.
A compare of the identities (\ref{eft1}) and (\ref{eft2}) shows that
\begin{equation} \label{eft3}
\langle P \omega(\xi \otimes b), P \omega(\eta \otimes c)\rangle_{\overline \cale_t} P(1_B \rtimes 1)
= \langle P(\xi \otimes b), P(\eta \otimes c) \rangle_{\cale_t} P(1_B \rtimes 1).
\end{equation}
By the uniqueness of the $\rho_P(B)$-factor (Lemma \ref{lemmauniquexp}), we see that
$\omega$ respects the inner product.
\end{proof}



\begin{theorem}   \label{theoremdelta}
$\delta^S$ and $\gamma^S$ are isomorphisms which are inverses to each other.
\end{theorem}

\begin{proof}
By Lemmas \ref{lemmadelta} and \ref{lemmagamma} we have $\gamma^S \delta^S(\cale,T) = (\mathsf{G} \mathsf{F}(\cale),\D (T \otimes 1)\D)$.
By the proof of Proposition \ref{propositionequivalenceCD} there is an isomorphism $\kappa: \cale \rightarrow \mathsf{G} \mathsf{F}(\cale)$
of $S$-Hilbert $A,B$-bimodules, where $\kappa = W \circ f$.
Since $(\cale,T)$ is a cycle, $k:= a T P - a P T$ is a compact operators on $\cale$ for $a \in A$ and $P \in X$.
Consequently,
\begin{eqnarray*}
&&  a \kappa^{-1} \D (T \otimes 1) \D \kappa(\xi)  =  \kappa^{-1} \sum_{P \in X}a  P TP(\xi) \otimes P(1) \\
&=& \kappa^{-1} \sum_{P \in X}  P(a  T(\xi) + k(\xi)) \otimes P(1)
= a T(\xi) + k(\xi). 
\end{eqnarray*}
This shows that $\kappa^{-1} \D(T \otimes 1)\D \kappa$ is a compact perturbation of $T$, and hence $\gamma^S \delta^S = \id$.
Similarly, one checks $\delta^S \gamma^S = \id$. 
%
\end{proof}


\begin{theorem}   \label{corollaryGreen}
Let $S$ be a finite unital inverse semigroup. Then there exists a
Green--Julg isomorphism $\mu^S$
determined by the commutative diagram
$$\begin{xy}
\xymatrix{ KK^S(\C,A) \ar[r]^{\delta^S} \ar[drr]_{\mu^S}  &
\widehat{KK^S}(C_0(X), A
\rtimes E)  \ar[r]^{\widehat{\mu^S}} &  K( (A \rtimes E) \widehat \rtimes S) \ar[d]^{\gamma_*} \\
&&  K(A \rtimes S)   . }
\end{xy}$$
\end{theorem}

\begin{proof}
Since $S$ is finite,
Paterson's groupoid $\calg_S$ \cite{0913.22001} associated to $S$ is finite and Hausdorff,
and
we may
choose $X= \calg^{(0)}$ as an example for an universal space for proper actions by $\calg_S$.
In this simple case, the Baum--Connes map for groupoids by Tu \cite{0932.19005} becomes an isomorphism
$$\mu^{\calg_S}:KK^{\calg_S}(C_0(X),A) \longrightarrow K(A \rtimes \calg_S).$$

In \cite{burgiKKrdiscrete} we have shown that $\widehat{KK^S}(A,B)$ and $KK^{\calg_S}(A,B)$ are isomorphic, and use
this, together with an isomorphism $B \widehat \rtimes S \cong B \rtimes \calg_S$ for $S$-algebras $B$ by Quigg and Sieben \cite{0992.46051},
to translate the last Baum--Connes isomorphism for groupoids to a Baum--Connes isomorphism
$\widehat{\mu^{S}}$ as in the above diagram.
%
The down arrow in the diagram is induced by
an isomorphism $\gamma:(A \rtimes E) \widehat \rtimes S \longrightarrow A \rtimes S$
by Khoshkam and Skandalis in \cite[Theorem 6.2]{1061.46047}.
The map $\delta^S$ of the diagram is the map of Lemma \ref{lemmadelta}, which is an isomorphism by
Theorem \ref{theoremdelta}.
\end{proof}

\bibliographystyle{plain}
\bibliography{references}

\end{document}